\documentclass[11pt,a4paper, final, twoside]{article}
\usepackage{amsmath}
\usepackage{fancyhdr}
\usepackage{amsthm}
\usepackage{amsfonts}
\usepackage{amssymb}
\usepackage{amscd}
\usepackage{latexsym}
\usepackage{amsthm}
\usepackage{graphicx}
\usepackage{graphics}
\usepackage{afterpage}
\usepackage[colorlinks=true, urlcolor=blue,  linkcolor=black, citecolor=black]{hyperref}
\usepackage{color}
\usepackage{bm}
\usepackage{tcolorbox}
\renewcommand{\thefootnote}{}
\newcommand{\HRule}{\rule{\linewidth}{1pt}}

\newtheorem{thm}{Theorem}[section]

\theoremstyle{proposition}
\newtheorem{prop}{Proposition}[section]

\newtheorem{conjecture}[thm]{Conjecture}

\theoremstyle{definition}

\theoremstyle{remark}

\numberwithin{equation}{section}
\newcommand{\pkn}{p_{k,n}}
\newcommand{\Pkn}{P_{k,n}}
\newcommand{\pin}{p_{i,n}}
\newcommand{\fact}{\,!}
\newcommand{\E}{\mathrm{e}}
\newcommand{\st}{\ensuremath{^\mathrm{st}}}
\newcommand{\nd}{\ensuremath{^\mathrm{nd}}}
\newcommand{\rd}{\ensuremath{^\mathrm{rd}}}
\renewcommand{\th}{\ensuremath{^\mathrm{th}}}

\newcommand{\for}{\quad\text{for}\quad}
\newcommand{\with}{\quad\text{with}\quad}

\newcommand{\folgt}{\ \Longrightarrow\ }
\newcommand{\half}{\ensuremath{\textstyle\frac{1}{2}}}
\newcommand{\quarter}{\ensuremath{\textstyle\frac{1}{4}}}
\newcommand{\Exp}[1]{\mathbb{E}\,[#1]}
\newcommand{\Var}[1]{\mathbb{V}\,[#1]}
\newcommand{\Gnsq}{G_n^{\,2}}
\newcommand{\ntoinf}{n\to\infty}
\newcommand{\X}{\bm{X}}
\newcommand{\N}{\mathbb{N}}
\newcommand{\beq}{\begin{equation}}
\newcommand{\eeq}{\end{equation}}

\begin{document}
\hyphenpenalty=100000

\begin{flushright}

{\Large \textbf{A new discrete distribution arising from a generalised\\[1mm] random game and its asymptotic properties}}\\[5mm]
{\large \textbf{R.~Fr\"{u}hwirth$^\mathrm{*1}$\footnote{\emph{*Corresponding author: E-mail: rudolf.fruehwirth@oeaw.ac.at}}, R.~Malina$^\mathrm{2}$ and W. Mitaroff\,$^\mathrm{1}$}}\\[1mm]
$^\mathrm{1}${\footnotesize \it Institute of High Energy Physics, Austrian Academy of Sciences, Vienna, Austria}\\[1mm] $^\mathrm{2}${\footnotesize \it CH-8330 Pf\"{a}ffikon (ZH), Switzerland}
\end{flushright}

\noindent\HRule\\[0mm]

%

\noindent{\Large \textbf{Abstract}}\\[4mm]
\fbox{%
\begin{minipage}{5.4in}{\footnotesize
The rules of a game of dice are extended to a ``hyper-die'' with $n\in\mathbb{N}$ equally probable faces, numbered from 1 to $n$. We derive recursive and explicit expressions for the probability mass function and the cumulative distribution function of the gain $G_n$ for arbitrary values of $n$. A numerical study suggests the conjecture that for $n \to \infty$ the expectation of the scaled gain $\Exp{H_n}=\Exp{G_n/\sqrt{n}\,}$ converges to $\sqrt{\pi/\,2}$.
The conjecture is proved by deriving an analytic expression of the expected gain $\Exp{G_n}$.
An analytic expression of the variance of the gain $G_n$ is derived by a similar technique. Finally,  it is proved that $H_n$ converges weakly to the Rayleigh distribution with scale parameter~1.}\end{minipage}}\\[3mm]
\footnotesize{\it{Keywords:} Random game; Mlynar distribution; Expected gain; Asymptotic behaviour; Weak\\ convergence; Rayleigh distribution}\\[1mm] 

\afterpage{
\fancyhead{} \fancyfoot{} 
\fancyfoot[R]{\footnotesize\thepage}
}

\small

\section{Introduction}
\label{sec:intro}

The German popular science journal ``Bild der Wissenschaft'' features a monthly column, ``Heinrich Hemme's Cogito'', posing a mathematical or logical puzzle. The November 2015 column presented the rules of a game of dice, called ``Mlyn\'a\v{r}'' (``Miller'' in Czech)~\cite{Hemme}, allegedly popular in Bohemia, but actually invented by the author~\cite{Hemme1}. The puzzle asked for calculating a player's average or expected gain.

The game is played with a standard die having faces numbered 1 to 6. 
Each player throws the die up to five times and collects a gain or a penalty $G_6$
according to the following rules:
\begin{itemize}
\item If $1\st$ throw shows a [1], the gain is $G_6=1$ point, and the player stops. Else: 
\item  \vspace{-5pt} 
If $2\nd$ throw shows [1] or [2], the gain is $G_6=2$ points, the player stops. Else:
\item  \vspace{-5pt} 
If $3\rd$ throw shows [1] to [3], the gain is $G_6 = 3$ points, the player stops. Else:
\item  \vspace{-5pt} 
If $4\th$ throw shows [1] to [4], the gain is $G_6 = 4$ points, the player stops. Else:
\item  \vspace{-5pt} 
If $5\th$ throw shows [1] to [5], the gain is $G_6 = 5$ points, the player stops. Else:
\item  \vspace{-5pt} 
Without a further throw, the gain is $G_6 = 6$ points, and the player stops.
\end{itemize}

Assuming a fair die, the probability of each face showing up is ${1}/{6}$. 
Hence, the probabilities $p_{k,6}=\mathbb{P}\,(G_6 = k)$ of a player gaining $k$ points can be easily computed, see Eqs.~(\ref{equ:pkrec})--(\ref{equ:pkpro2}).
The expected value of the gain, i.e., the average number of points a player will collect in the long run, is defined by
\begin{equation}
\label{equ:exp6}
\Exp{G_6} = \sum_{k = 1}^6 \, k\cdot p_{k,6} = 2.7747 \: \mathrm{points,}
\end{equation}
where the $p_{k,6}$ are calculated by one of Eqs.~(\ref{equ:pkrec})--(\ref{equ:pkpro2}) for $n = 6$. 
They are plotted in Fig.~\ref{fig:pkn}(a). 

\section{Generalisation} 
\label{sec:general} 
\setcounter{footnote}{0}
\renewcommand{\thefootnote}{\arabic{footnote}}

An interesting one-parametric discrete distribution arises if the game is generalised by replacing the ordinary die with a ``hyper-die'' having $n\in\N$ equally probable faces and formalising the rules in the following way. 

Let $\X=\{X_k,\ k\geq 1\}$ be a stochastic process such that the $X_k$ are independent and identically distributed according to the discrete uniform distribution on the set $\{1,2,\ldots,n\}$. Let $G_n$ be the gain associated with $\X$\!, and $\tau$ be a stopping time with respect to $\X$.\footnote{For an elementary definition of the stopping time, see~\cite{Stopping}.} Then the game is defined by the following rule:
\beq\label{equ:defgame}
\text{\bf Rule of the game:}\quad\mathrm{For}\ k=1,2,\ldots,n:\quad X_k\leq k \folgt \tau=k,\ G_n=k.
\eeq
It follows that the probability mass 
function $\pkn=\mathbb{P}\,(G_n=k)$ of the gain $G_n$ is recursively defined by 
\begin{equation}
\label{equ:pkrec}
p_{1,n} = \frac{1}{n} \, , \quad  
\pkn = \left( 1 - \sum_{i = 1}^{k-1} \, \pin \right) \cdot \frac{k}{n} = 
p_{k-1,n} \cdot k \cdot \left( \frac{1}{k - 1} - \frac{1}{n} \right) 
\for 2 \le k \le n.
\end{equation}
\indent
The last probability is $p_{n,n} = 1 - \sum_{i = 1}^{n-1} \pin$, and the sum of all probabilities is $\sum_{k = 1}^{n} \, \pkn = 1$ as required. 
The probabilities in Eq.~(\ref{equ:pkrec}) can also be expressed explicitely by 
\begin{equation}
\label{equ:pkpro1}
\pkn = \frac{1}{n} \cdot \prod_{i = 2}^k \, i \cdot \left( \frac{1}{i - 1} - \frac{1}{n} \right) 
\for 2 \le k \le n,
\end{equation}
or more conveniently by
\begin{equation}
\label{equ:pkpro2}
\pkn = \frac{k}{n^k} \cdot \frac{(n - 1)\fact}{(n - k)\fact}
\for 1 \le k \le n.
\end{equation}
While the recursive definition in Eq.~(\ref{equ:pkrec}) is well suited to numerical computations (see below and Section~\ref{sec:empirical}), the expression in Eq.~(\ref{equ:pkpro2}) is the most useful one for the further theoretical analysis (see Sections~\ref{sec:proof} and~\ref{sec:variance}). 

To the best of our knowledge, this distribution has not been described in the literature before. We propose to call it the ``Mlynar distribution'' with parameter $n$. As another example, the probability mass function for the case $n = 25$  is plotted in Fig.~\ref{fig:pkn}(b).

\begin{figure}[h]
\centering
\scalebox{0.90}{\includegraphics[trim=10mm 0mm 0mm 0cm,clip]{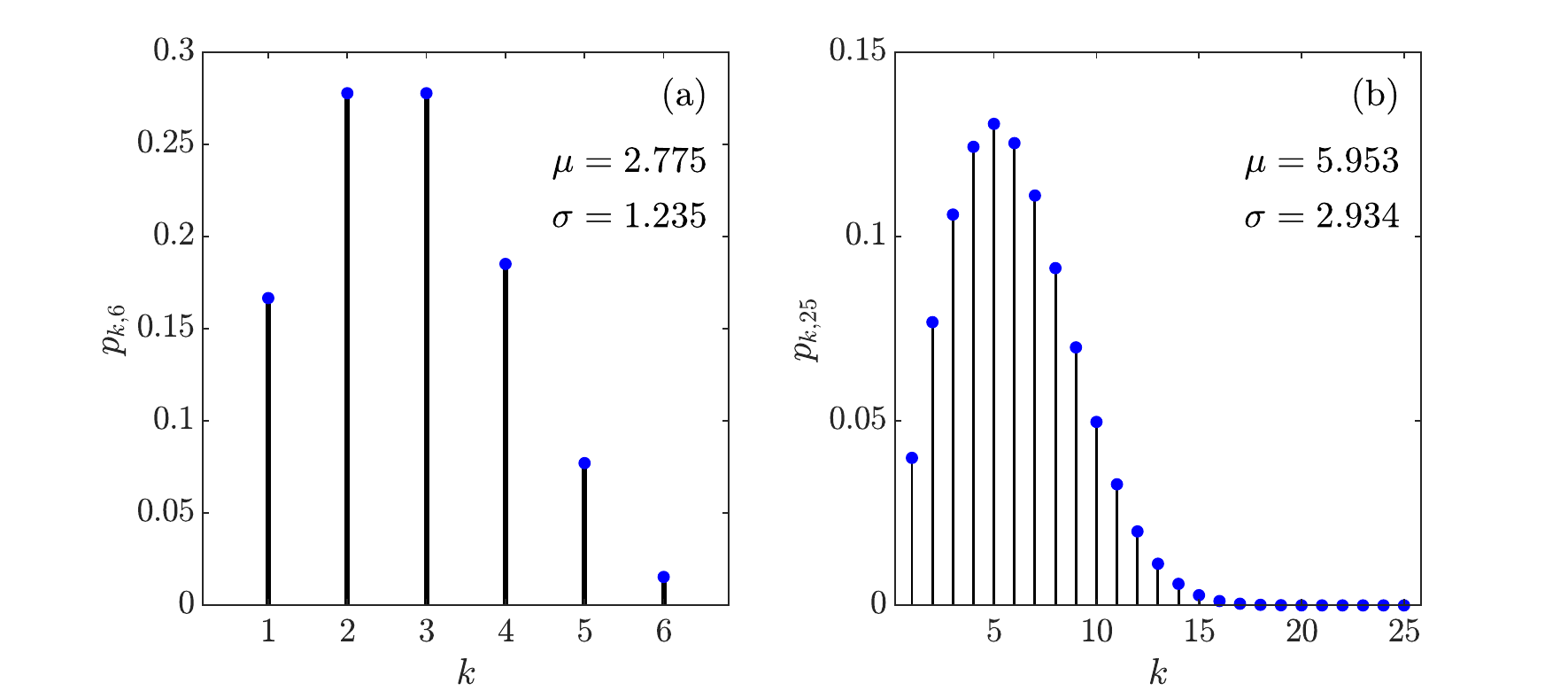}} 
\caption[]{Probabilities $\pkn$ for (a) $n = 6$, (b) $n=25$ } 
\label{fig:pkn} 
\end{figure}

The cumulative distribution function (cdf) $P_n(x)=\mathbb{P}\,(G_n\leq x)$ of $G_n$ is obtained by summing over the probabilities in Eq.~(\ref{equ:pkpro2}) up to $\chi=\lfloor x \rfloor$, the largest integer that does not exceed $x$:
\begin{equation}
\label{equ:Gncdf}
P_n(x)=\sum_{i=1}^\chi\,\pin = 1-\frac{(n-1)\fact}{n^\chi\,(n-1-\chi)\fact} \with \chi=\lfloor x \rfloor.
\end{equation}

The mode of the distribution can be determined by rewriting Eq.~(\ref{equ:pkrec}) in the following form:
\begin{equation}\label{equ:pkmode}
\pkn = p_{k-1,n}\cdot\phi_{k,n},\with\phi_{k,n}= \frac{nk-(k-1)\,k}{nk-n}\for 2 \le k \le n.
\end{equation}
It is easily verified that $\phi_{k,n}$ is a monotonically decreasing function of $k$ for fixed $n$. Hence, the probability mass function $\pkn$ is unimodal and the mode $m$ is the largest integer $k$ such that $\phi_{k,n} \ge 1$.

If $n$ can be expressed as $n = (m - 1) \cdot m$, then $\phi_{m,n} = 1$, and both values
$m - 1$ and $m$ are modes; see for example Fig.~\ref{fig:pkn}(a) 
where $n = 6 = 2 \cdot 3$. Otherwise, the mode $m$ is the floor of the positive solution of the quadratic equation $k^2-k=n$:
\begin{equation}\label{equ:pkmode1}
m=\left\lfloor\half+\sqrt{\quarter+n}\right\rfloor.
\end{equation}
As an example, see Fig.~\ref{fig:pkn}(b) 
where $n = 25$ and $m = \lfloor 5.5249\ldots \rfloor = 5$.

The expectation of $G_n$ is given by
\begin{equation}
\label{equ:funcf}
g(n) = \Exp{G_n} = \sum_{k = 1}^n \, k\cdot\pkn=(n - 1)\fact \cdot \sum_{k = 1}^n \, \frac{k^2}{n^k \cdot (n - k)\fact}.
\end{equation}
Using Eq.~(\ref{equ:pkrec}), the function $g(n)$ can be calculated in double precision floating point arithmetic without numerical problems for $n$ up to $10^{15}$. As can be seen in Fig.~\ref{fig:pkn}(b), the $\pkn$ fall off very quickly in the tail of the distribution already for $n=25$. If $K(n)$ is the index in the sum in Eq.~(\ref{equ:funcf}) beyond which addition of another term has no effect because of rounding errors, then $K(n)\leq 8.5\sqrt{n}$ for $1\leq n\leq 10^{15}$.

As guessed  initially and proved 
below in Section~\ref{sec:proof}, $g(n)\sim C\sqrt{n}$ as $\ntoinf$, with $C\in\mathbb{R}$. Therefore, a scaled expectation $h(n)$ is defined as 
\begin{equation}
\label{equ:funch}
h(n) = \Exp{H_n} = \frac{g(n)}{\sqrt{n}}, \with  H_n = \frac{G_n}{\sqrt{n}}.
\end{equation}
$H_n$ will be called the scaled gain in the following. The functions $g(n)$ and $h(n)$ are plotted 
in Fig.~\ref{fig:fnhn25}(a) and (b), respectively, in the range $1\le n\le25$. 

\begin{figure}[h!t]
\centering
\scalebox{0.90}{\includegraphics[trim=10mm 2mm 0mm 0cm,clip]{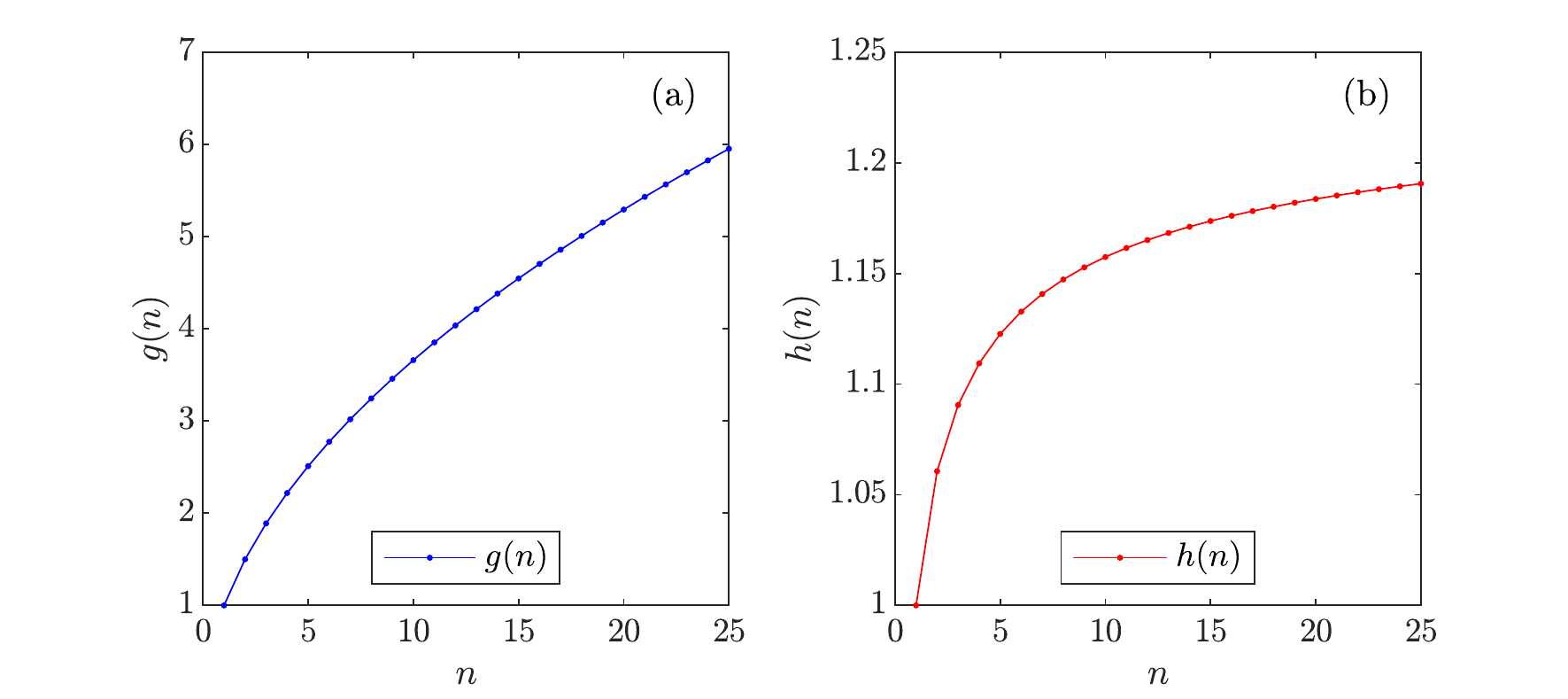}}
\caption[]{(a) $g(n)$ up to $n = 25$. (b) $h(n)$ up to $n = 25$.} 
\label{fig:fnhn25} 
\end{figure}

\section{Empirical conjecture} 
\label{sec:empirical} 

The function $h(n)$ is monotonically increasing and apparently converges to a 
non-zero constant value, as shown in Fig.~\ref{fig:hnbig}(a) up to $n = 10^{10}$, with a logarithmic scale on the abscissa. 
A surprising  observation reveals this constant to be  empirically equal 
(within our numerical precision) to $\sqrt{\pi / 2}$. This inspires us to the following conjecture: 
\begin{conjecture}\label{conj:1} 
The scaled expectation value $h(n)$ converges for $n \to \infty$ as 
\begin{equation}
\label{equ:conj1}
\lim_{n \to \infty} h(n) = \lim_{n \to \infty} \Exp{H_n} 
= \sqrt{\frac{\pi}{2}} = 1.2533 \dots .
\end{equation}
\end{conjecture}
In order to corroborate the conjecture, the difference function 
\begin{equation} 
\label{equ:deln} 
\Delta(n) = \sqrt{\frac{\pi}{2}} - h(n) 
\end{equation} 
is calculated and plotted in Fig.~\ref{fig:hnbig}(b) up to $n = 10^{10}$ with a logarithmic scale on the abscissa. 
\begin{figure}[t]
\centering
\scalebox{0.90}{\includegraphics[trim=0mm 0cm 0mm 0mm,clip]{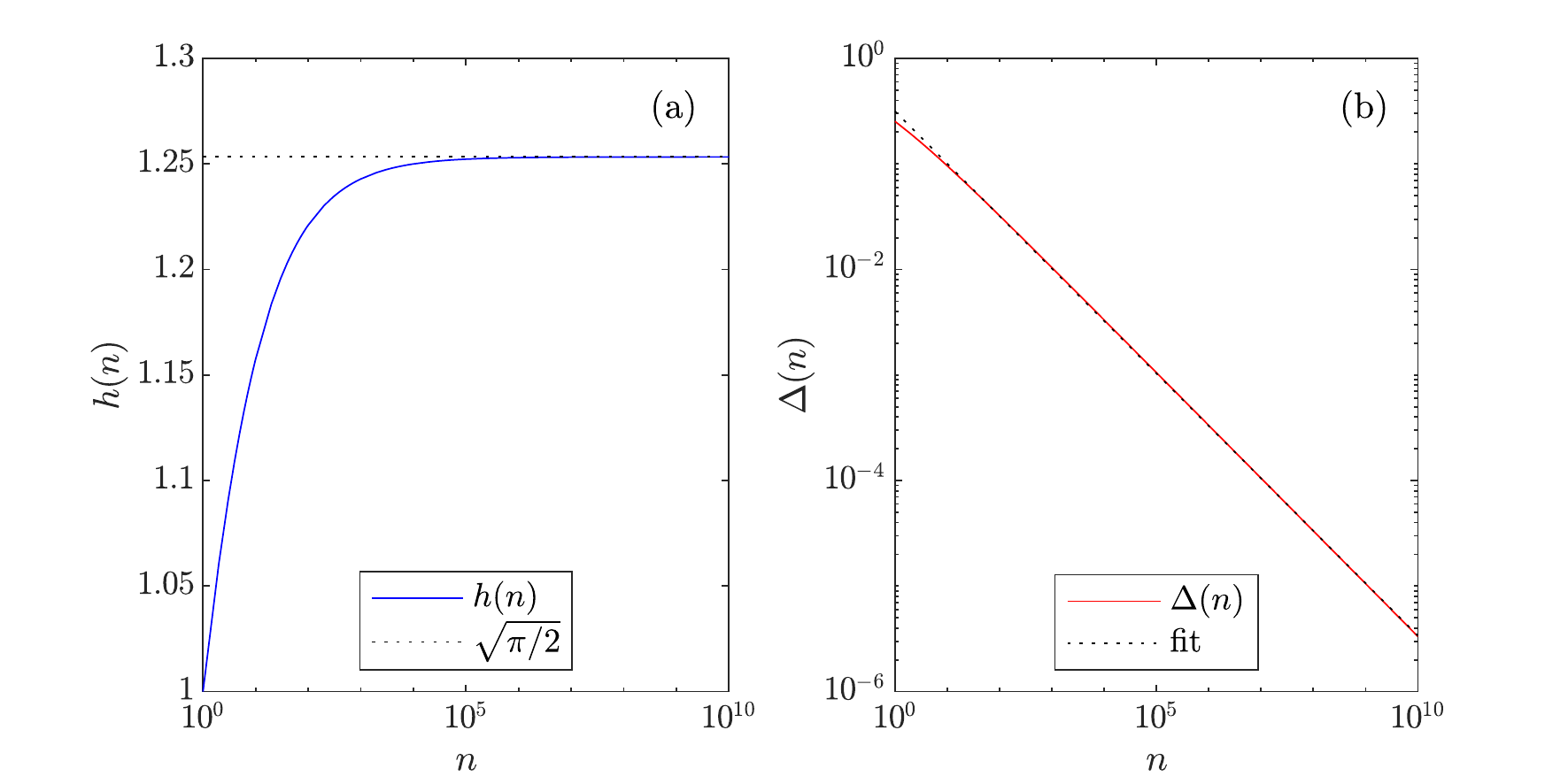}} 
\caption[]{(a) Function $h(n)$. (b) $\Delta(n)$ in Eq.~(\ref{equ:deln}) and the fitted power function.} 
\label{fig:hnbig} 
\end{figure}
Its behaviour above $n = 10$ shows approximate linearity in double log scale, 
suggesting an ansatz $\Delta(n) = c \cdot n^{- s}$ or 
\begin{equation}
\label{equ:ansatz}
\log_{10} \Delta(n) = \alpha + \beta \cdot \log_{10} n, \with \alpha=\log_{10} c,\quad\beta=-s.
\end{equation} 
The intercept $\alpha$ and the slope $\beta$ have been estimated by linear regression~\cite{Lyons} on 10 data points ($n_\mathrm{df}=8$) situated at 
$ \log_{10} n = 1, 2, \ldots, 10$, 
yielding the estimates
\begin{equation}
\label{equ:fits} 
\tilde{\alpha} = -0.4996 \pm 0.0068, \ 
\tilde{\beta} =  - \tilde{s} = - 0.4970 \pm 0.0011,\ 
\tilde{c}=10^{\tilde{\alpha}}=0.3165\pm 0.0050,
\end{equation}
where the errors represent $1\,\sigma$. 
 
Assuming that above ansatz holds beyond $n = 10^{10}$, 
the extrapolation $n \to \infty$ adds strong evidence 
for the exact validity of Eq.~(\ref{equ:conj1}). 
The asymptotic behaviour of $h(n)$ can thus be parametrized, 
within the fitted accuracy, by 
\begin{equation} 
\label{equ:hnfit} 
h(n) = \frac{1}{\sqrt{n}} \cdot \Exp{G_n} \approx 
\sqrt{\frac{\pi}{2}} - \Delta(n), \with \Delta(n) = 0.3165 \cdot n^{- 0.4970}.
\end{equation} 
It is remarkable that within $2.7\,\sigma$, $\tilde{s} \approx 0.5$, 
hinting at $\Delta(n) \propto {1}/{\sqrt{n}}$. 

\section{Proof of the conjecture}
\label{sec:proof}
In order to prove Conjecture~\ref{conj:1}, we first derive an explicit expression for the expected gain $g(n)=\Exp{G_n}$ for fixed $n$. 
\begin{prop} 
For any $n\in\N$, $g(n)=\Exp{G_n}$ is given by
\begin{equation}
\label{equ:thm1}
g(n) = \sum_{k=1}^n\,k\cdot\pkn = \E^n\, n^{-n}\, \Gamma(n+1,n) -1,
\end{equation}
where $\Gamma(a,x)$ is the upper incomplete Gamma function as defined in~\cite{Alpha} and~\cite[Eq.~(6.5.3)]{AS}.
\begin{proof}
The sum in Eq.~(\ref{equ:thm1}) can be rewritten in the following way:
\begin{equation}
\label{equ:rewrite}
\sum_{k = 1}^n\, k\cdot\pkn= \sum_{k = 1}^n\,\sum_{i=k}^n\,\pin.
\end{equation}
\indent
It is easy to see that in the double sum on the right hand side $p_{1,n}$ occurs exactly once, $p_{2,n}$ occurs exactly twice, and so on, up to $p_{n,n}$ which occurs exactly $n$ times. The same is true for the sum on the left hand side.

The resulting double sum in Eq.~(\ref{equ:rewrite}) can be evaluated in two steps, with the help of~\cite{Alpha}. In the first step we obtain
\begin{equation}
\label{equ:thm1step1}
\Pkn=\sum_{i=k}^n\,\pin=\frac{(n-1)\fact}{n^{k-1}\,(n-k)\fact}.
\end{equation}
The second step yields
\begin{equation}
\label{equ:thm1step2}
\Exp{G_n} = \sum_{k=1}^n\,\Pkn = \E^n\, n^{-n}\, \Gamma(n+1,n) -1=g(n).
\end{equation}
This concludes the proof.
\end{proof}
\end{prop}

The asymptotic behaviour of $h(n)$ as $\ntoinf$ is given by the following proposition. 
\begin{prop} 
Let $h(n)$ be as in Eq.~(\ref{equ:funch}). Then, as $\ntoinf$,
\begin{equation}
\label{equ:thm2}
h(n)\sim\sqrt{\frac{\pi}{2}}-\frac{1}{3\sqrt{n}}\quad \Longrightarrow\quad \lim_{\ntoinf} h(n)=\sqrt{\frac{\pi}{2}}.
\end{equation}
\begin{proof}
The asymptotic behaviour of $\Gamma(n+1,n)$ as $\ntoinf$ is given in~\cite[Eq.~(6.5.35)]{AS}:
\begin{equation}
\Gamma(n+1,n)\sim \E^{-n}\,n^n\,\left(\sqrt{\frac{n\,\pi}{2}}+\frac{2}{3}+\frac{\sqrt{2\,\pi}}{24\,\sqrt{n}}+\cdots\right).
\end{equation}
It follows that as $\ntoinf$
\begin{equation}
\label{equ:fnasym}
g(n)\sim -1+\sqrt{\frac{n\,\pi}{2}}+\frac{2}{3}+\frac{\sqrt{2\,\pi}}{24\,\sqrt{n}}+\cdots=
\sqrt{\frac{n\,\pi}{2}}-\frac{1}{3}+\frac{\sqrt{2\,\pi}}{24\,\sqrt{n}}+\cdots.
\end{equation}
Dividing by $\sqrt{n}$ and omitting terms that are $\mathcal{O}(1/n)$ yields Eq.~(\ref{equ:thm2}).
\end{proof}
\end{prop}
It should be noted that the empirical result in Eq.~(\ref{equ:hnfit}) is in very good agreement with the assertion of the proposition. 

\section{Variance and asymptotic distribution}
\label{sec:variance}

The variance of $G_n$ can be computed as $\Var{G_n}=\Exp{\Gnsq} - \Exp{G_n}^2$. The following proposition gives an explicit expression for $\Var{G_n}$.

\begin{prop}
The variance of $G_n$ is given by
\begin{equation}
\label{equ:thm3}
\Var{G_n}=2\,n-\Exp{G_n}-\Exp{G_n}^2.
\end{equation}
\begin{proof}
The expectation of $\Gnsq$ can be rewritten in the following form:
\begin{equation}
\Exp{\Gnsq}\,=\,\sum_{k=1}^n k^{2}\cdot\pkn=\sum_{k=1}^n\sum_{i=1}^k\, (2\,i-1)\cdot\pkn =
\sum_{k=1}^n\,(2\,k-1)\cdot\Pkn.
\end{equation}
It is not difficult to verify that in all sums $p_{1,n}$ occurs exactly once, $p_{2,n}$ occurs exactly four times, and so on, up to $p_{n,n}$ which occurs exactly $n^2$ times. With the help of~\cite{Alpha} and using Eq.~(\ref{equ:thm1step2}), the last sum evaluates to
\beq
\Exp{\Gnsq}=2\cdot\sum_{k=1}^n\,k\cdot\Pkn - \sum_{k=1}^n\,\Pkn=2\,n-\Exp{G_n}.
\eeq
Subtracting the squared expectation yields the proposition.
\end{proof}
\end{prop}
The function $v(n)$ is defined by $v(n)=\Var{H_n}=\Var{G_n/\sqrt{n}\,}=\Var{G_n}/n$.
Its asymptotic behaviour as $\ntoinf$ is described by the following proposition.
\begin{prop}
As $\ntoinf$,
\beq
v(n)\sim 2-\frac{\pi}{2}-\sqrt{\frac{\pi}{18\,n}}+\cdots
\quad \Longrightarrow\quad \lim_{\ntoinf} v(n)=2-\frac{\pi}{2}.
\eeq
\begin{proof}
The assertion follows by an elementary calculation from Eqs.~(\ref{equ:fnasym}) and~(\ref{equ:thm3}).
\end{proof}
\end{prop}

A well-known distribution with mean $\sqrt{\pi/\,2}$ and variance $2-\pi/\,2$ ist the Rayleigh distribution with scale parameter $\sigma=1$~\cite{Johnson}. Its cdf $R(x)$ is given by
\beq
\label{equ:Rayleigh}
R(x)=\begin{cases}
1-\exp(-x^2/\,2)& \for x\geq 0,\\
0& \for x<0.
\end{cases}
\eeq
The following proposition shows that this distribution is indeed the asymptotic distribution of $H_n$ for $\ntoinf$.

\begin{prop}
The sequence ($H_n,\ n\in\N$) converges weakly (in distribution) to a random variable with the cdf $R(x)$ in Eq.~(\ref{equ:Rayleigh}):
\beq
\lim_{\ntoinf} Q_n(x)=1-\exp(-x^2/\,2), \for x\geq 0,
\eeq
where $Q_n(x)$ is the cdf of $H_n$ (see also Eq.~(\ref{equ:Gncdf})):
\beq
\label{equ:Hncdf}
Q_n(x) = P_n(x\sqrt{n})=1-\frac{(n-1)\fact}{n^{\chi}\,(n-1-\chi)\fact} \with \chi=\lfloor x\sqrt{n} \rfloor.
\eeq 
\end{prop}
\begin{proof}
As $P_n(x)$ is increasing, the following inequality holds:
\beq\label{equ:ineqPn}
P_n(x\sqrt{n}-1) \leq P_n(\lfloor x\sqrt{n} \rfloor) \leq P_n(x\sqrt{n}).
\eeq
We rewrite $P_n$ in the form
\beq
P_n(u)=1-\exp(L(u)),\with L(u)=\ln\Gamma(n)-u\ln n-\ln\Gamma(n-u).
\eeq
According to~\cite[Eq.~(6.1.41)]{AS}, $\ln\Gamma(z)$ can be approximated by
\beq
\ln\Gamma(z) \sim (z-\half)\cdot\ln(z)-z+\half\cdot\ln(2\,\pi), \for z\to\infty.
\eeq
With this approximation, we obtain:
\beq
L(u)=\ln\left(n\right)\,\left(n-\half\right)-u+\ln\left(n-u\right)\,\left(u-n+\half\right)-u\,\ln\left(n\right).
\eeq
Using~\cite{SymTB} and~\cite{Alpha} for the calculation of the limits in Eq.~(\ref{equ:limits}), as well as the fact that $L(u)$ is decreasing, we obtain that
\beq\label{equ:limits}
\lim_{\ntoinf} L(x\sqrt{n}-1)=\lim_{\ntoinf} L(x\sqrt{n})=-\frac{x^2}{2}\folgt
\lim_{\ntoinf} L(\lfloor x\sqrt{n}\rfloor)=-\frac{x^2}{2}.
\eeq
It follows that
\beq\label{equ:conv}
\lim_{\ntoinf} Q_n(x)=\lim_{\ntoinf} P_n(\lfloor x\sqrt{n}\rfloor)=1-\exp(-x^2/\,2).
\eeq
This concludes the proof.
\end{proof}
The convergence in distribution is illustrated by Fig~\ref{fig:conv}. It shows that already for $n=10^4$ it is virtually impossible to visually distinguish $Q_n(x)$ and $R(x)$. A numerical study shows that the maximum absolute difference $d(n)=\max_x |Q_n(x)-R(x)|$ can be parametrized by $d_n\approx 0.44/\sqrt{n}$ for $n\ge100$.

\begin{figure}[t]
\centering
\scalebox{0.90}{\includegraphics[trim=0mm 0mm 0mm 0mm,clip]{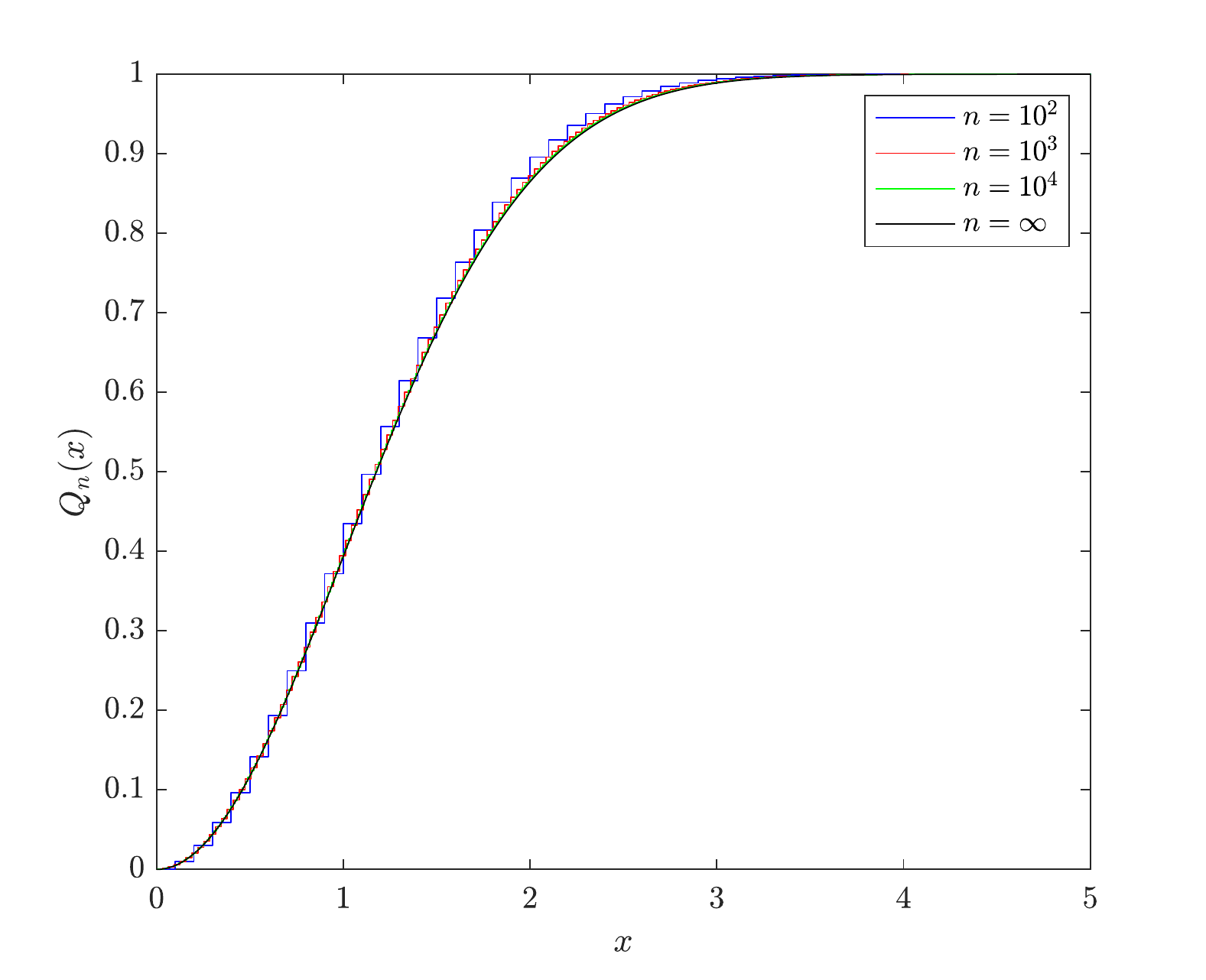}}
\caption[]{Distribution functions of $H_n$ and the limiting Rayleigh cdf.}
\label{fig:conv}
\end{figure}

\section{Conclusions} 
\label{sec:conclude} 

Inspired by an elementary puzzle, we have analyzed the generalisation of a game of dice to a ``hyper-die'' with an arbitrary number $n$ of faces. The gain of a player is described by a novel probability mass function and its corresponding cumulative distribution function. We propose to call the distribution the ``Mlynar distribution'' with parameter $n$.

An empirical study has led to the conjecture that the {expectation of the scaled gain} converges to the constant $\sqrt{\pi/\,2}$. A proof of the conjecture has been found, based on an analytic expression of the gain. A simple expression for the variance of the gain has been derived as well, {thereby proving that the variance of the scaled} {gain converges to $2-\pi/2$}. 

Finally, it has been proved that the scaled gain converges weakly (in distribution) to the Rayleigh distribution with scale parameter~1.

\section*{Acknowledgments}
We thank H.~Hemme for some background information about the invention of the original game. We also thank the anonymous reviewers for useful comments.

\section*{Authors' contributions}
This work was carried out in collaboration of all authors. Authors RM and WM defined the generalized game, mainly contributed Sects. 2 and 3, and wrote the first draft of the manuscript. Author RF fully contributed Sects. 4 and 5, and completed and finalized the manuscript. All authors have read and approved the submitted manuscript.

\section*{Competing interests}
The authors declare that no competing interests exist.

\noindent\scriptsize\----------------------------------------------------------------------------------------------------------------------------------------------
\copyright \it 2021  R. Fr\"uhwirth, R. Malina and W. Mitaroff. This is an Open Access article distributed under the terms of the Creative Commons Attribution License
\href{http://creativecommons.org/licenses/by/2.0}{http://creativecommons.org/licenses/by/2.0},  which permits unrestricted use, distribution, and reproduction in any medium,
provided the original work is properly cited.

\end{document}